\documentclass[11pt,reqno,a4paper]{amsart}
\usepackage[utf8]{inputenc}
\usepackage{amsmath,amssymb,amsthm,enumerate}
\usepackage{hyperref} 
\newtheorem{theorem}{Theorem}
\newtheorem{proposition}[theorem]{Proposition}
\newtheorem{corollary}[theorem]{Corollary}
\newtheorem{lemma}[theorem]{Lemma}
\theoremstyle{definition}
\newtheorem{definition}[theorem]{Definition}

\newtheorem{example}[theorem]{Example}

\newcounter{que}
\newtheorem{question}[que]{Question}
\newcommand{\R}{\mathbb R}
\newcommand{\N}{\mathbb N}
\newcommand{\setsep}{:\;}
\newcommand{\dens}{\operatorname{dens}}
\newcommand{\cf}{\operatorname{cf}}
\newcommand{\ind}{\operatorname{ind}}
\def\C{\mathcal{C}}
\def\F{\mathcal{F}}

\def\U{\mathcal{U}}
\def\NN{\mathcal{N}}

\begin{document}
\title[Large separated sets of unit vectors]{Large separated sets of unit vectors in Banach spaces of continuous functions}

\author{Marek C\'uth}
\author{Ond\v{r}ej Kurka}
\author{Benjamin Vejnar}
\email{cuth@karlin.mff.cuni.cz}
\email{kurka.ondrej@seznam.cz}
\email{vejnar@karlin.mff.cuni.cz}
\address[M.~C\' uth, O.~Kurka, B.~Vejnar]{Charles University, Faculty of Mathematics and Physics, Department of Mathematical Analysis, Sokolovsk\'a 83, 186 75 Prague 8, Czech Republic}
\address[O.~Kurka]{Institute of Mathematics of the Czech Academy of Sciences, \v{Z}itn\'a 25, 115 67 Prague 1, Czech Republic}

\subjclass[2010]{46B20, 46B04, 46E15 (primary), and 54D30, 46B26 (secondary)}

\thanks{M.~C\'uth and B.~Vejnar are junior researchers in the University Centre for Mathematical Modelling, Applied Analysis and Computational Mathematics (MathMAC). M.~C\'uth and B.~Vejnar were supported by the Research grant GA\v{C}R 17-04197Y. O.~Kurka was supported by the Research grant GA\v{C}R 17-04197Y and by RVO: 67985840}

\keywords{equilateral set, $(1+)$-separated set, Banach space of continuous functions}

\begin{abstract}The paper concerns the problem of whether a nonseparable $\C(K)$ space must contain a set of unit vectors whose cardinality equals the density of $\C(K)$, and such that the distances between every two distinct vectors are always greater than one. We prove that this is the case if the density is at most continuum, and that for several classes of $\C(K)$ spaces (of arbitrary density) it is even possible to find such a set which is $2$-equilateral; that is, the distance between every two distinct vectors is exactly 2.
\end{abstract}
\maketitle

In this paper we deal with distances between unit vectors in Banach spaces. For $r\in\R$ a set $A$ in a Banach space $X$ is said to be \emph{$r$-separated}, \emph{$(r+)$-separated} and \emph{$r$-equilateral} if $\|v_1 - v_2\| \geq r$,  $\|v_1 - v_2\| > r$ and $\|v_1 - v_2\| = r$ for distinct $v_1, v_2\in A$, respectively.

A natural question considered in the literature is whether, given a Banach space, there is a big equilateral set or there is a big $(1+)$-separated set in the unit sphere of the space. Our investigation is motivated mainly by the recent papers \cite{KaniaKochanek,Koszmider,MV15}, where this question has been addressed also for nonseparable Banach spaces of the form $\C(K)$, where $\C(K)$ is the Banach space of all continuous functions on a compact space $K$ considered with the supremum norm.

Let us summarize what is known. By \cite[Theorem 2.6]{MV15}, the existence of an $r$-separated set of cardinality $\kappa$ in the unit ball of $\C(K)$ for some $r>1$ is equivalent to the existence of a $2$-equilateral set of cardinality $\kappa$ in the unit sphere of $\C(K)$. Hence, in $\C(K)$ spaces it suffices to consider the problem of the existence of big $(1+)$-separated sets and $2$-equilateral sets in the sphere. By \cite{Koszmider}, it is undecidable in ZFC whether there always exists an uncountable equilateral set in a nonseparable $\C(K)$ space. On the other hand, by \cite{MV15} and \cite{KaniaKochanek}, there always exists an uncountable $(1+)$-separated set in the unit sphere of a nonseparable $\C(K)$ space. Moreover, if $K$ is nonmetrizable and not perfectly normal, then there exists an uncountable $2$-equilateral set in the unit sphere \cite{MV15}, and if $K$ is perfectly normal, then there exists a $(1+)$-separated set in the unit sphere of cardinality equal to the density of $\C(K)$, see \cite{KaniaKochanek} (by density we understand the minimal cardinality of a dense subset). It is mentioned in \cite[page 40]{KaniaKochanek} that the following is a ``tantalising problem'' left open by the authors.

\begin{question}\label{q:jednaPlus}Let $K$ be a compact Hausdorff space with $\kappa:=\dens \C(K)>\omega$. Does there exist a $(1+)$-separated set in the unit sphere of $\C(K)$ of cardinality $\kappa$?
\end{question}

We were not able to answer this question. However, we prove that for quite many classes of compact spaces the answer is positive.

Let us emphasize that all compact spaces of our considerations are supposed to be infinite and Hausdorff. Recall that if $K$ is a compact Hausdorff space then the weight of $K$ (denoted by $w(K)$) is equal to the density of $\C(K)$. Our main results read as follows.

\begin{theorem}\label{t:main1}Let $K$ be a compact space such that $w(K)$ is at most continuum. Then the unit ball of $\C(K)$ contains a $(1+)$-separated set of cardinality $w(K)$.
\end{theorem}

\begin{theorem}\label{t:main2}Let $K$ be a compact space such that at least one of the following conditions is satisfied.
\begin{itemize}
    \item[(1)] $K$ contains two disjoint homeomorphic compact spaces of the same weight as $K$.
    \item[(1')] $K$ is homogeneous, or homeomorphic 
    to a compact convex set in a locally convex space.
    \item[(2)] $K$ is a compact line (that is, a linearly ordered space with the order topology).
\end{itemize}
Then the unit ball of $\C(K)$ contains a $2$-equilateral set of cardinality $w(K)$.
\end{theorem}

Theorem~\ref{t:main1} follows from the slightly more general Theorem~\ref{t:jednaPlus}, proved in Section~\ref{section:onePlus}, and Theorem~\ref{t:main2} summarizes the most important new results from Section~\ref{sekceDvaSeparovanost}. Let us note that (1') is a consequence of (1)
(see Lemma~\ref{l:konvexKpkt}). Let us mention that there are more situations where the conclusion of Theorem~\ref{t:main2} is true; namely, it suffices to use certain known results in order to find a big left(right)-separated sequence in $K$ which, as observed already in \cite{MV15}, is sufficient for the existence of a big $2$-equilateral set in the unit ball of $\C(K)$. The cases of when this is possible are summarized in Corollary~\ref{c:densEqualsWeight}.

Note that separable compact spaces and first countable compact spaces are of weight at most continuum and so Theorem \ref{t:main1} applies. Indeed, it is a classical result, see \cite{deGroot} or \cite[Theorem 3.3]{handbookCardinalFunctions}, that for a regular topological space $X$ we have $w(X)\leq 2^{\dens X}$; hence, separable compact spaces have weight at most continuum. If $K$ is a first countable compact space, then by the famous inequality of Arhangel'skii, see \cite{A69} or \cite[Theorem 7.1 and 7.3]{handbookCardinalFunctions}, it has cardinality at most $\mathfrak{c}$. Hence, the weight is at most $\mathfrak{c}$ as well.

Let us remark that our results generalize the result from \cite[Theorem~B(ii)]{KaniaKochanek} mentioned above. Indeed, if $K$ is perfectly normal, it is first countable and we may apply Theorem~\ref{t:main1}.

Our results naturally suggest certain problems/conjectures which we summarize in the last section of this paper. 

\section{Preliminaries}

The notation and terminology is standard; for the undefined notions see \cite{fhhz} for Banach spaces, \cite{eng} for topology and \cite{kunen} for set theory. By $\mathfrak{c}$ we denote the cardinality of continuum. If $X$ is a set and $\kappa$ a cardinal, we denote by $[X]^{\kappa}$ the set of all subsets of $X$ of cardinality $\kappa$. For a cardinal $\kappa$ we denote by $\cf(\kappa)$ its cofinality. If $X$ is a topological space and $A\subset X$, $\overline{A}$ stands for the closure of $A$ and $\dens X$ stands for the density of $X$. The closed unit ball of a Banach space $X$ is denoted by $B_X$. 
In our proofs we use without mentioning the well-known fact that for any compact space $K$ we have
\[
w(K) = \dens \C(K) = \omega + \min \{|\F|\setsep \F\subset \C(K)\text{ separates the points of }K\}.
\]

Let us mention some easy facts which we will use later. First, it is easy to see that if a compact space $K$ is metrizable, then the unit ball of $\C(K)$ contains a $2$-equilateral set of cardinality $w(K)$. For a proof one may, for example, use Lemma~\ref{t:dens}. The other facts are formulated as lemmas below.

\begin{lemma}\label{l:podmnoziny}Let $K$ be a compact space and let $L$ be a closed subset or a continuous image of $K$. If the unit ball of $\C(L)$ contains a $(1+)$-separated (resp. $2$-equilateral) set of cardinality $\kappa$, then the unit ball of $\C(K)$ contains a $(1+)$-separated (resp. $2$-equilateral) set of cardinality $\kappa$.
\end{lemma}
\begin{proof}If $L\subset K$ then we conclude the proof using Tietze's extension theorem. If $\varphi:K\to L$ is continuous and surjective we realize that $\C(L)$ is isometric to a subspace of $\C(K)$ by the mapping $f\mapsto f\circ\varphi$.
\end{proof}
\begin{lemma}[\protect{\cite[Theorem 2.9(i)]{MV15}}]\label{l:nuldimenzionalni}If a compact space $K$ contains a zero-dimensional compact subspace of weight $w(K)$, then the unit ball of $\C(K)$ contains a $2$-equilateral set of cardinality $w(K)$.
\end{lemma}
\begin{lemma}[\protect{\cite[Theorem 2.6]{MV15}}]\label{l:staci1+eps}Let $K$ be a compact space and $\kappa$ be an infinite cardinal. Then the unit ball of $\C(K)$ contains a $2$-equilateral set of cardinality $\kappa$ if and only if it contains a $(1+\varepsilon)$-separated set of cardinality $\kappa$ for some $\varepsilon > 0$.
\end{lemma}

\section{Sets separated by more than 1}\label{section:onePlus}
\begin{theorem}\label{t:jednaPlus}Let $K$ be a compact space. Then the unit ball of $\C(K)$ contains a $(1+)$-separated set of cardinality $w(K)$ or it contains a $2$-equilateral set of cardinality $\mathfrak{c}$.
\end{theorem}
\begin{proof}We may without loss of generality assume that $K$ is nonmetrizable, see e.g. Theorem~\ref{t:dens}. Given $f\in\C(K)$ and $x\in K$, we say that $t\in\R$ is a \emph{local maximum of $f$ at $x$} if $f(x)=t$ and there exists an open neighbourhood $U$ of $x$ such that $f(y) \leq t$ for every $y\in U$. Let us consider the following condition inspired by the proof of \cite[Theorem 4.11]{KaniaKochanek}:
\begin{equation*}\tag{P1}\label{podminkaKK}
\begin{split}
&\forall x,y\in K,x\neq y, \quad \exists f\in \; B_{\C(K)}:  f(x)=1,\\ & f(z)=-1\text{ for every $z$ in some neighborhood of }y\\ &\text{and }0\text{ is not a local maximum of }f\text{ at any point}.
\end{split}
\end{equation*}

First, let us assume that condition \eqref{podminkaKK} holds. Take a maximal $(1+)$-separated family $\F$ (with respect to inclusion) of norm-one functions such that $0$ is not a local maximum of any $f\in\F$ at any point. We \emph{claim} that the cardinality of $\F$ equals $w(K)$.

Suppose for contradiction that the cardinality of $\F$ is less than $w(K)$. It follows that $\F$ does not separate the points of $K$. Thus, for some pair of distinct points $x,y\in K$ and every $g\in\F$ we have $g(x)=g(y)$. Since \eqref{podminkaKK} holds, we may find a norm-one function $f\in\C(K)$ such that $f(x)=1$, $f(z)=-1$ in some neighborhood $U$ of $y$ and 0 is not a local maximum of $f$ at any point. Fix any $g\in \F$. If $g(x)=g(y)\neq 0$, then 
$$\|f-g\|\geq \max\{|1 - g(x)|,|-1 - g(y)|\} > 1.$$
If $g(x)=g(y)=0$, since 0 is not a local maximum of $g$ at $y$, there is $y'\in U$ with $g(y')>0$ and we have
$$\|f-g\|\geq|f(y') - g(y')|=|-1-g(y')| > 1.$$
Therefore, we have $\|f-g\|>1$ for any $g\in\F$ which is a contradiction with the  maximality of $\F$.

On the other hand, let us assume that \eqref{podminkaKK} does not hold. Then there is a pair of distinct points $x, y\in K$ which witnesses the negation of \eqref{podminkaKK}. Pick a function $f\in B_{\C(K)}$ such that $f(x)=1$ and $f(z)=-1$ for every $z$ in some neighborhood of $y$. By the choice of the pair $x,y$, we know that every $t\in (-1,1)$ is a local maximum of $f$ at some point $x_t\in K$. Indeed, if $t\in (-1,1)$ is not a local maximum of $f$ at any point, then we can easily modify the function $f$ in such a way that $0$ is not a local maximum at any point and this would contradict the choice of the pair $x,y$.

Hence, for every $t\in(-1,1)$, $f(x_t)=t$ and there exists a neighborhood $U_t$ of $x_t$ with $f(z)\leq t$ for every $z\in U_t$. We have $x_s\notin U_t$ for $s > t$ and thus $x_t\notin \overline{\{x_s\setsep s > t\}}$. Therefore, for every $t\in(-1,1)$, we may pick a function $f_t\in B_{\C(K)}$ with $f_t(x_t)=1$ and $f_t(x_s)=-1$ for every $s>t$. Then $\{f_t\setsep t\in(-1,1)\}$ is a $2$-equilateral set of cardinality $\mathfrak{c}$.
\end{proof}
\section{Equilateral sets}\label{sekceDvaSeparovanost}

In this Section we summarize positive results concerning big $2$-equilateral sets in the unit ball of $\C(K)$. The first subsection summarizes the situations where known techniques may be applied, the second subsection contains new arguments.

\subsection{Big right(left)-separated families}

In this subsection we summarize the situations when we may obtain a $2$-equilateral set of cardinality $\kappa$ in the unit ball of $\C(K)$ using certain arguments from the literature and from \cite{MV15}. Recall that given a  topological space $X$ we say that a family $\{x_\alpha\setsep \alpha<\kappa\}$ is \emph{right-separated}, resp.  \emph{left-separated} if $x_\alpha\notin \overline{\{x_\beta\setsep \alpha< \beta < \kappa\}}$, resp. $x_\alpha\notin \overline{\{x_\beta\setsep \beta < \alpha\}}$ for $\alpha < \kappa$. In \cite{MV15}, the following is proved for $\kappa = \omega_1$, the same argument may be applied to a general ordinal (the proof is so short that we give it here for the convenience of the reader).

\begin{proposition}[\protect{\cite[proof of Theorem 2.9(ii) and (iii)]{MV15}}]\label{prop:separated}Let $K$ be a compact space such that there exists a family $\{x_\alpha\setsep \alpha<\kappa\}\subset K$ which is left-separated or right-separated. Then the unit ball of $\C(K)$ contains a $2$-equilateral set of cardinality $\kappa$.
\end{proposition}
\begin{proof}If the family is right-separated, for each $\alpha<\kappa$ we pick a norm-one function $f_\alpha$ such that $f_\alpha(x_\alpha) = 1$ and $f_\alpha(x_\beta) = -1$ for $\alpha<\beta<\kappa$. If the family is left-separated, for each $\alpha<\kappa$ we pick a norm-one function $f_\alpha$ such that $f_\alpha(x_\alpha) = 1$ and $f_\alpha(x_\beta) = -1$ for $\beta<\alpha$. Then $\{f_\alpha\setsep \alpha < \kappa\}$ is a $2$-equilateral set.
\end{proof}

There is a connection between the existence of left-separated and right-separated families and certain cardinal invariants. Given a topological space $X$, the \emph{hereditary density} of $X$ is defined as $hd(X) = \sup\{\dens(Y)\setsep Y\subset X\}$ and the \emph{hereditary Lindel\"of degree} of $X$ is $hL(X) = \sup \{L(Y)\setsep Y\subset X\}$, where $L(Y)$ is the Lindel\"of degree of $Y$, i.e. the minimal cardinal $\kappa\geq\omega$ such that each  open cover of $Y$ has a subcover of cardinality $\kappa$. It is well-known, see e.g. \cite[page 4]{T14}, that we have 
\[\begin{split}
hd(X) & = \sup\{\kappa\setsep\text{ there is a left-separated family of cardinality $\kappa$ in }X\},\\
hL(X) & = \sup\{\kappa\setsep\text{ there is a right-separated family of cardinality $\kappa$ in }X\}.
\end{split}\]
Hence, if $K$ is compact space such that $hd(K)$ or $hL(K)$ is greater than or equal to a successor cardinal $\kappa$, by Proposition~\ref{prop:separated}, the unit ball of $\C(K)$ contains a $2$-equilateral set of cardinality $\kappa$.

In the following we observe some conditions under which a left (right)-separated family of cardinality $\kappa$ exists. Those are most probably well-known, the proof is short so we include it for the convenience of the reader. Recall that for a closed subset $L$ in a compact space $K$ the \emph{pseudo-character} $\psi(L,K)$ is the minimal cardinal $\kappa$ such that $L$ is the intersection of $\kappa$ many open sets in $K$. 

\begin{lemma}\label{t:dens}Let $K$ be a compact space and $\kappa$ be a cardinal.
\begin{enumerate}[(i)]
	\item\label{densita} If there exists a set $A\subset K$ with $\dens A \geq \kappa$, then there exists a left-separated family of cardinality $\kappa$.
    \item\label{char} If there exists a closed subset $L\subset K$ with $\psi(L,K) \geq \kappa$, then there exists a right-separated family of cardinality $\kappa$.
\end{enumerate}
\end{lemma}
\begin{proof}If there is $A\subset K$ with $\dens A\geq \kappa$, we inductively find points $\{x_\alpha\setsep \alpha < \kappa\}\subset A$ such that $x_\alpha\notin \overline{\{x_\beta\setsep \beta < \alpha\}}$.

Let us assume that there exists a closed subset $L\subset K$ with $\psi(L,K) \geq \kappa$. Then $L$ is not the intersection of less than $\kappa$ open sets. We shall inductively find points $x_\alpha$ and open sets $U_\alpha$ for $\alpha < \kappa$ such that $L\subset U_{\alpha}$ and
\begin{equation}\label{eq:bodyProCharacter2}x_\alpha\in \bigcap_{\beta < \alpha}(U_{\beta}\setminus\overline{U_{\alpha}}).\end{equation}

Pick $x_0\notin L$ and an open $U_0\supset L$ with $x_{0}\notin\overline{U_0}$. Having chosen $x_\beta$ and $U_\beta$ for every $\beta < \alpha$, we pick a point $x_\alpha\in \bigcap_{\beta < \alpha}U_{\beta}\setminus L$ and then we find $U_\alpha\supset L$ such that $x_\alpha\notin \overline{U_\alpha}$. In this way we have picked all the $x_\alpha$'s and, by \eqref{eq:bodyProCharacter2}, we have $x_\beta\notin\overline{\{x_\alpha\setsep \alpha > \beta\}}$ for every $\beta < \kappa$. \end{proof}

Note that we may apply the approach above to several classes of compact spaces which include also classes studied in functional analysis. For a survey about Valdivia and Corson compacta we refer to \cite{kalendaSurvey}, for information about Eberlein compacta to \cite{fabCervenaKniha}. Let us recall that a cardinal $\kappa$ is a \emph{strong limit cardinal} if $2^\lambda < \kappa$ whenever $\lambda < \kappa$.

The following was suggested to us by O.~Kalenda.

\begin{lemma}\label{l:contImageOfValdivia} Let $K$ be a continuous image of a Valdivia compact space $L$. Then there exists a set $A\subset K$ with $\dens A \geq w(K)$.
\end{lemma}
\begin{proof}Let $D\subset L$ be a dense $\Sigma$-subset of $L$ and let $\phi:L\to K$ be a continuous surjection. Pick $D':=\{d_i\setsep i\in I\}\subset D$ such that $\{\phi(d_i)\setsep i\in I\}$ is dense in $\phi(D)$ and  $|I| = \dens \phi(D)$. Then $\overline{D'}$ is a Valdivia compact and $D'$ is its dense $\Sigma$-subset; hence, by \cite[Lemma 3.4]{kalendaSurvey}, $\dens D' = w(\overline{D'})$, and we have $\dens \phi(D) = |I| \geq \dens D' = w(\overline{D'})$. Since $\phi(\overline{D'}) = K$, we have $w(\overline{D'})\geq w(K)$. Thus, it suffices to put $A = \phi(D)$.
\end{proof}

Let us end up with a list of several situations where the results of this subsection may be applied. Note that the last case (that is, $K$ homeomorphic to $L\times L$) is essentially proved already in \cite[Corollary 2.11]{MV15}.

\begin{corollary}\label{c:densEqualsWeight}Let $K$ be a compact space. Suppose at least one of the following conditions holds.
\begin{itemize}
\item $K$ is a continuous image of a Valdivia compact space (e.g. $K$ is metrizable, Eberlein, Corson
);
\item $w(K)$ is a strong limit cardinal;
\item $K$ is a connected continuous image of a linearly ordered compact space;
\item $K$ is homeomorphic to $L\times L$ for a compact space $L$.
\end{itemize}
Then the unit ball of $\C(K)$ contains a $2$-equilateral set of cardinality $w(K)$.
\end{corollary}
\begin{proof}In the first three cases we may apply Proposition~\ref{prop:separated} and Lemma~\ref{t:dens}(i). If $K$ is a continuous image of a Valdivia compact space, then we use Lemma~\ref{l:contImageOfValdivia}.

If $w(K)$ is a strong limit cardinal, then $\dens K = w(K)$. Indeed, it is a classical result, see \cite{deGroot} or \cite[Theorem 3.3]{handbookCardinalFunctions}, that for a regular topological space $X$ we have $w(X)\leq 2^{\dens X}$; hence, if $\dens K < w(K)$ we would get $w(K) < w(K)$, a contradiction.

If $K$ is a connected continuous image of a linearly ordered compact space then, by \cite{Tr}, we have $\dens K = w(K)$.

If $K = L\times L$ for a compact space $L$, we use Proposition~\ref{prop:separated}, Lemma~\ref{t:dens}(ii) and the fact that $\psi(\Delta L,L\times L)\geq w(L\times L)$, see e.g. \cite[page 16]{handbookCardinalFunctions}.
\end{proof}

\subsection{New techniques}

Another class of compact spaces where a $2$-equilateral set of cardinality $w(K)$ in the unit ball of $\C(K)$ exists is given by the following result.

\begin{theorem}\label{t:disjKopie}Let $K$ be a compact space. Then the unit ball of $\C(K\times \{0,1\})$ contains a $2$-equilateral set of cardinality $w(K)$.
\end{theorem}
\begin{proof}By Lemma~\ref{t:dens}, we assume that $K$ is nonmetrizable. By Lemma~\ref{l:staci1+eps}, it is sufficient to find a $\tfrac{3}{2}$-separated set of cardinality $w(K)$.

For $f\in\C(K\times 2)$, where $2:=\{0,1\}$, consider the following condition:
\begin{equation*}\tag{P2}\label{podminkaK2}
	\forall z\in K:\; |f(z,0)| < \tfrac{1}{2} \implies f(z,1)=-1.
\end{equation*}
Take a maximal $\frac{3}{2}$-separated family $\F$ (with respect to inclusion) of norm-one functions satisfying condition \eqref{podminkaK2}. We \emph{claim} that the cardinality of $\F$ equals $w(K)$.
Suppose for contradiction that the cardinality of $\F$ is less than $w(K)$. It follows that $\F$ does not separate the points of $K\times\{0\}$. Thus, for some pair of distinct points $x,y\in K$ and every $g\in\F$ we have $g(x,0)=g(y,0)$. Now, consider any norm-one function $f\in\C(K\times 2)$ satisfying condition \eqref{podminkaK2} such that $f(y,0)=-1$ and $f(x,0)=f(x,1)=1$. Such a function exists because we may pick any $\tilde{f}\in B_{\C(K)}$ with $\tilde{f}(x)=1=-\tilde{f}(y)$ and take any continuous extension of a function defined on disjoint closed sets $K\times\{0\}$, $\{(x,1)\}$ and $\tilde{f}^{-1}\left([-\tfrac{1}{2},\tfrac{1}{2}]\right)\times \{1\}$ in the obvious way, that is, $f(z,0)=\tilde{f}(z)$ for every $z\in K$, $f(x,1)=1$ and $f(z,1) = -1$ for $z\in \tilde{f}^{-1}\left([-\tfrac{1}{2},\tfrac{1}{2}]\right)$.

Fix any $g\in \F$. If $g(x,0)=g(y,0) \geq \tfrac{1}{2}$, then
$$\|f-g\|\geq |-1 - g(y,0)| = 1 + g(y,0)\geq \tfrac{3}{2}.$$
If $g(x,0)=g(y,0) \leq -\tfrac{1}{2}$, then
$$\|f-g\|\geq |1 - g(x,0)| = 1 - g(x,0)\geq \tfrac{3}{2}.$$
If $|g(x,0)| < \tfrac{1}{2}$, then since $g$ satisfies \eqref{podminkaK2} we have
$$\|f-g\|\geq |f(x,1) - g(x,1)| = 1 - g(x,1) = 2.$$
Therefore, we have $\|f-g\|\geq \tfrac{3}{2}$ for any $g\in\F$ which is a contradiction with the maximality of $\F$.
\end{proof}

\begin{corollary}\label{c:disjKopie}Let $K$ be a compact space which contains two disjoint homeomorphic compact spaces of weight $w(K)$. Then the unit ball of $\C(K)$ contains a $2$-equilateral set of cardinality $w(K)$.
\end{corollary}

\begin{proof}The statement follows immediately from Theorem~\ref{t:disjKopie} and Lemma~\ref{l:podmnoziny}.
\end{proof}

Even though there are known examples of spaces that do not satisfy the hypothesis of Corollary~\ref{c:disjKopie} (e.g. the space $[0,\kappa]$ for any cardinal $\kappa$), Corollary~\ref{c:disjKopie} seems to give a new and strong sufficient condition for the existence of a big equilateral set. This is witnessed by the following interesting consequences. 
\begin{lemma}\label{l:konvexKpkt}
Let $K$ be a compact convex subset of a locally convex space $E$. Then $K$ contains two disjoint subsets homeomorphic to itself.
\end{lemma}
\begin{proof}
Let $ x, y \in K$ be two distinct points and let $ x^{*} \in E^{*} $ be such that $ x^{*}(y - x) > 0 $. It is sufficient to show that
$$ (1-\lambda) x + \lambda K \quad \textrm{and} \quad (1-\lambda) y + \lambda K $$
are disjoint for a small enough $ \lambda \in (0, 1] $. Assuming the opposite for some $ \lambda $, we obtain that there are $ u, v \in K $ such that
$$ (1-\lambda) x + \lambda u = (1-\lambda) y + \lambda v, $$ 
which implies
$$ \sup x^{*}(K) - \inf x^{*}(K) \geq x^{*}(u) - x^{*}(v) = \frac{1-\lambda}{\lambda} x^{*}(y - x). $$
Therefore, any $ \lambda \in (0, 1] $ satisfying $ \frac{1-\lambda}{\lambda} x^{*}(y - x) > \sup x^{*}(K) - \inf x^{*}(K) $ works.
\end{proof}

\begin{corollary}Let $K$ be a compact space which is homeomorphic 
to a compact convex set in a locally convex space. Then the unit ball of $\C(K)$ contains a $2$-equilateral set of cardinality $w(K)$.
\end{corollary}

\begin{proof}
We use Corollary~\ref{c:disjKopie} and Lemma~\ref{l:konvexKpkt}.
\end{proof}

Before proving another corollary of Theorem~\ref{t:disjKopie}, let us formulate the following easy observation.

\begin{lemma}\label{l:bodVelkeVahy}Let $K$ be a compact space. Then there exists a point $x\in K$ such that $w(U) = w(K)$ for every neighborhood $U$ of $x$.
\end{lemma}
\begin{proof}Suppose for contradiction that for every point $x$ there exists an open neighborhood $U_x$ of $x$ with $w(U_x)<w(K)$. By compactness, there are points $x_1,\ldots,x_n\in K$ such that $U_{x_1}\cup\ldots\cup U_{x_n} = K$; hence, $w(K) = w(U_{x_1}\cup\ldots\cup U_{x_n})\leq w(U_{x_1}) + \ldots + w(U_{x_n}) < w(K)$, a contradiction.
\end{proof}

\begin{corollary}Let $K$ be a homogeneous compact space. Then the unit ball of $\C(K)$ contains a $2$-equilateral set of cardinality $w(K)$.
\end{corollary}
\begin{proof}
By Lemma~\ref{l:bodVelkeVahy}, there exists a point $x\in K$ such that $w(U)=w(K)$ for every neighborhood of $x$. Pick $y\in K\setminus\{x\}$ and a homeomorphism $h:K\to K$ with $h(x) = y$.

Find open neighborhoods $U$ and $V$ of $x$ and $y$ respectively such that $\overline{U}\cap\overline{V}=\emptyset$ and $h(U)=V$. This is indeed possible since we may pick arbitrary neighborhoods $U_0$ and $V_0$ of $x$ and $y$ respectively such that $\overline{U_0}\cap\overline{V_0}=\emptyset$ and put $U:=h^{-1}(V_0)\cap U_0$, $V:=h(U)$.

Now, $\overline{U}$ and $\overline{V}$ are homeomorphic, disjoint and $w(\overline{U}) = w(K)$; hence, we may apply Corollary~\ref{c:disjKopie}.
\end{proof}

The next corollary of Theorem~\ref{t:disjKopie} is based on a variant of the Ramsey theorem for higher cardinalities.

\begin{definition}Let $\kappa$ and $\lambda$ be cardinals. By writing
\[\kappa \to (\lambda)_2^2\]
we mean that the following statement is true: for every set $X$ of cardinality $\kappa$ and for every $F:[X]^2\to \{0,1\}$ there exists a subset $Y$ of $X$ of cardinality $\lambda$ such that $F|_{[Y]^2}$ is constant.
\end{definition}

\begin{corollary}\label{c:sipky}Let $K$ be a compact space and let $\kappa$ be the weight of $K$. If $\lambda$ is a cardinal with $\kappa \to (\lambda)_2^2$, then the unit ball of $\C(K)$ contains a $2$-equilateral set of cardinality $\lambda$.
\end{corollary}
\begin{proof}
By Theorem~\ref{t:disjKopie}, in the unit ball of $\C(K\times\{0,1\})$ there exists a $2$-equilateral set $X$ of cardinality $\kappa$. Consider the mapping $F:[X]^2\to \{0,1\}$ such that $F(\{f,g\}) = 0$ if and only if there exists a point $x\in K\times\{0\}$ with $|(f-g)(x)| = 2$. If there is a set $Y\subset X$ of cardinality $\lambda$ such that $F|_{[Y]^2}\equiv 0$ then $\{f|_{K\times\{0\}}\setsep f\in Y\}$ is a $2$-equilateral set in the unit ball of $\C(K\times\{0\})$. Otherwise, there is a set $Y\subset X$ of cardinality $\lambda$ such that $F|_{[Y]^2}\equiv 1$ and $\{f|_{K\times\{1\}}\setsep f\in Y\}$ is a $2$-equilateral set.
\end{proof}

As a corollary, we may obtain the following result.

\begin{corollary}\label{c:oJednoNize}
Let $K$ be a compact space with $w(K) \geq (2^{<\kappa})^+$ for some cardinal $\kappa$. Then there exists a $2$-equilateral set of cardinality $\kappa$ in the unit ball of $\C(K)$.
\end{corollary}
\begin{proof}
It follows from Corollary~\ref{c:sipky} and the classical Erd\H os-Rado theorem (see e.g. \cite[Theorem 2.9]{handbookSipky}), which states that $(2^{<\kappa})^+ \to (\kappa)_2^2$ whenever $\kappa$ is an infinite cardinal.
\end{proof}

It deserves to be mentioned here that by the result of Terenzi \cite{T89}, if $E$ is any Banach space with $\dens E\geq (2^\mathfrak{c})^+$, then $E$ contains an equilateral set $S$ with $|S|\geq \mathfrak{c}^+$.

\begin{theorem}\label{t:LinearneUsporadane}Let $K$ be a linearly ordered compact space. Then the unit ball of $\C(K)$ contains a $2$-equilateral set of cardinality $w(K)$.
\end{theorem}
\begin{proof}Put $\kappa:=w(K)$. By Lemma~\ref{t:dens} we may suppose that $\lambda:=\dens(K)<\kappa$. Let $D\subset K$ be a dense set of cardinality $\lambda$. The cardinality of the system of open intervals $\{(a, b)\setsep a, b\in D, a < b\}$ is $\lambda$, hence it is not a base for $K$. Put
\[L:=\{x\in K\setsep \exists a<x: (a,x)=\emptyset\},\quad R:=\{x\in K\setsep \exists b>x: (x,b)=\emptyset\}.\]

We notice that if $x\in L$, then the set $V_x=\{y\in K\setsep y<x\}$ is a clopen subset of $K$, hence $|L|\leq w(K)$. Analogously, we get that $|R|\leq w(K)$. We \emph{claim} that either $L$ or $R$ is of cardinality $w(K)$. Indeed, assume the opposite case. For every $x\in L\setminus R$, there is $a_x < x$ with $(a_x,x)=\emptyset$ and $\mathcal{B}_x:=\{(a_x,b)\setsep b\in D, b > x\}$ is a neighborhood basis of $x$. Similarly, for every $x\in R\setminus L$ we find $b_x > x$ such that $\mathcal{B}_x:=\{(a,b_x)\setsep a\in D, a  < x\}$ is a neighborhood basis of $x$. Note that the points of $L\cap R$ are isolated. Then
\[
\mathcal{B} = \big\{(a, b)\setsep a, b\in D, a < b\big\}\cup \bigcup\big\{\mathcal{B}_x\setsep x\in L\triangle R\big\}\cup \big\{\{x\}\setsep x\in L\cap R\big\}\]
is of cardinality less then $\kappa$. Moreover, it is easy to see that $\mathcal{B}$ is a basis of $K$; hence, $w(K)<\kappa$, a contradiction.

Now, assume that the cardinality of $L$ is $\kappa$. For every $x\in L$ consider a continuous function $f_x$ defined as
\[f_x(y)=\begin{cases}1, & y\geq x,\\
-1, & y <x.
\end{cases}\]
Then $\{f_x\setsep x\in L\}$ is a $2$-equilateral set of cardinality $w(K)$. The case when $R$ is of cardinality $\kappa$ is similar.
\end{proof}

It is worth mentioning that the density or even the hereditary density and the hereditary Lindel\"of degree of a linearly ordered compact space can be less than the weight. This is witnessed e.g. by the Alexandrov double arrow space.

\section{Remarks and questions}

Up to our knowledge it is not known whether in the unit ball of a nonseparable Banach space $X$ there exists a $1$-separated set of cardinality equal to the density of $X$. However, if we consider only Banach spaces of the form $\C(K)$, this is easy. The proof of the following proposition is a straightforward modification of the proof of \cite[Theorem 4.11]{KaniaKochanek}. A special case of it is mentioned without a proof in \cite[Question 2.7 (1), (b)]{MV15}.

\begin{proposition}Let $K$ be a compact space. Then the unit ball of $\C(K)$ contains a $1$-separated set of cardinality $w(K)$.
\end{proposition}
\begin{proof}By Theorem~\ref{t:dens}, we may assume that $K$ is nonmetrizable. Take a maximal $1$-separated family $\F$ (with respect to inclusion) of norm-one functions. We \emph{claim} that the cardinality of $\F$ equals $w(K)$. Suppose for contradiction that the cardinality of $\F$ is less than $w(K)$. Then $\F$ does not separate the points of $K$. Thus, for some pair of distinct points $x, y\in K$ and every $g\in\F$ we have $g(x) = g(y)$. Find a norm-one function $f\in\C(K)$ such that $f(x) = 1 = -f(y)$, then we have $\|f-g\| \geq 1$ for every $g\in\F$; hence $\F\cup\{f\}$ is $1$-separated, which is a contradiction with the maximality of $\F$.
\end{proof}

We get easily from the results in the previous section that the situation is quite simple under GCH.

\begin{corollary}[GCH]\label{t:LimitWeight} Let $K$ be a compact space.
\begin{enumerate}
	\item If $w(K)$ is a limit cardinal, then the unit ball of $\C(K)$ contains a $2$-equilateral set of cardinality $w(K)$.
    \item If $w(K) = \kappa^+$ for an infinite cardinal $\kappa$, then the unit ball of $\C(K)$ contains a $2$-equilateral set of cardinality $\kappa$.
\end{enumerate}
\end{corollary}
\begin{proof}
The first statement follows immediately from Corollary~\ref{c:densEqualsWeight} because under GCH every limit cardinal is a strong limit cardinal.

Concerning the second statement, we may apply Corollary~\ref{c:oJednoNize}. Indeed, it is sufficient to notice that under GCH we have $2^{<\kappa} = \kappa$, which follows from the computation
\[
2^{<\kappa} = \sup\{2^\lambda\setsep \lambda < \kappa\}=\sup\{\lambda^+\setsep \lambda < \kappa\} = \kappa,
\]
where the first equality follows e.g. from \cite[Lemma I.13.17]{kunen} and the second from GCH.
\end{proof}

\begin{question}Does Corollary~\ref{t:LimitWeight} hold in ZFC?
\end{question}

Moreover, we do not know if it is possible to have an analogue of Koszmider's example \cite{Koszmider} for higher densities.

\begin{question}Let $\kappa \geq \omega_1$ be a cardinal. Does there (at least consistently) exist a compact space of weight $\kappa^+$ such that the unit sphere of $\C(K)$ does not contain a $2$-equilateral set of cardinality $\kappa^+$?
\end{question}



P. Koszmider proved \cite{Koszmider} that consistently there exists a nonmetrizable compact space $K$ without an uncountable equilateral set in $\C(K)$. By considering his construction in detail, it is quite easy to see that for Koszmider's example we have $\ind(K)\leq 2$, where $\ind(K)$ is the topological dimension (for a definition see e.g \cite[Chapter 7]{eng}). Since in zero-dimensional nonmetrizable compact spaces there always exists an uncountable $2$-equilateral set in the unit ball of $\C(K)$, it is of a certain interest to know what is the situation for compact spaces with dimension $1$. 

Modifying Koszmider's example, it is possible to obtain the following.

\begin{example}\label{e:last}It is relatively consistent with ZFC that there exists a nonmetrizable compact space $K$ with $\operatorname{ind}(K) = 1$ such that there does not exist an uncountable $2$-equilateral set in the unit ball of $\C(K)$. Moreover, the compact space $K$ is first countable, $w(K)=\omega_1$, it is hereditary separable and hereditary Lindel\"of.\end{example}

In the rest of this paper we outline the above mentioned modification of Koszmider's example. In order to shorten the notation, for a finite subset $N$ of $\omega$ and $s\in 2^N$, put $\NN_s:=\{x\in 2^\omega\setsep x|_N = s\}$. First, following the proof of \cite[Theorem 3.3]{Koszmider} (replacing the interval $[0,1]$ by the compact space $2^\omega$), one observes that it is sufficient to prove that consistently there are points $\{r_\xi\setsep \xi<\omega_1\}\subset 2^\omega$ and a sequence of functions $(f_\xi\setsep \xi<\omega_1)$, where $f_\xi:2^\omega\setminus\{r_\xi\}\to [-1,1]$ are continuous, such that given
\begin{itemize}
		\item[(a)] $m\in \N$,
        \item[(b)] a finite subset $N$ of $\omega$ and pairwise different sequences $s_1,\ldots,s_m\in 2^N$,
        \item[(c)] any sequence $(F_\alpha)_{\alpha<\omega_1}$ where $F_\alpha = \{\xi_1^\alpha,\ldots,\xi_m^\alpha\}$ are pairwise disjoint finite subsets of $\omega_1$ such that $r_{\xi_i^\alpha}\in\NN_{s_i}$ for every $1\leq i\leq m$ and every $\alpha < \omega_1$,
        \item[(d)] any $m$-tuple $\{q_1,\ldots,q_m\}$ of rational numbers from $[-1,1]$,
    \end{itemize}
 there are  $\alpha < \beta < \omega_1$, a finite subset $M\supset N$ of $\omega$ and sequences $(t_i^\alpha)_{1\leq i\leq m}$, and $(t_i^\beta)_{1\leq i\leq m}$ from $2^M$ such that for each $1\leq i\leq m$ we have:
 \begin{itemize}
 \item[(1)] $\NN_{t_i^\alpha}\cup \NN_{t_i^\beta}\subset \NN_{s_i}$ and $t_i^\alpha\neq t_i^\beta$,
 \item[(2)] $r_{\xi_i^\alpha}\in \NN_{t_i^\alpha}$ and $r_{\xi_i^\beta}\in \NN_{t_i^\beta}$,
 \item[(3)] $f_{\xi_i^\alpha}\restriction_{2^\omega\setminus(\NN_{t_i^\alpha}\cup \NN_{t_i^\beta})} = f_{\xi_i^\beta}\restriction_{2^\omega\setminus(\NN_{t_i^\alpha}\cup \NN_{t_i^\beta})}$,
 \item[(4)] $f_{\xi_i^\alpha}\restriction_{\NN_{t_i^\beta}} = q_i = f_{\xi_i^\beta}\restriction_{\NN_{t_i^\alpha}}$.
 \end{itemize}
 Now, similarly as in \cite[Section 4]{Koszmider}, by a forcing argument, we prove that consistently such points $\{r_\xi\setsep \xi< \omega_1\}\subset 2^\omega$ and a sequence of functions $(f_\xi\setsep \xi<\omega_1)$ exist.  Fix any points $\{r_\xi\setsep \xi< \omega_1\}\subset 2^\omega$. The forcing notion $\mathbb P$ consists of triples $(N_p,F_p,\F_p)$ such that
 \begin{itemize}
 \item[(1)] $N_p\in[\omega]^{<\omega}$,
 \item[(2)] $F_p$ is a finite subset of $\omega_1$ such that $\{r_\xi\restriction_{N_p}\setsep \xi\in F_p\}$ are pairwise different sequences,
 \item[(3)] $\F_P = \{f_p^\xi\setsep \xi\in F_p\}$,
 \item[(4)] $f_p^\xi:2^\omega\setminus\NN_{r_\xi\restriction_{N_p}}\to [-1,1]$ is a \emph{rationally piecewise constant function} for each $\xi\in F_p$ (i.e. for every $s\in 2^{N_p}$ with $s\neq r_\xi\restriction_{N_p}$ there is a rational number $q_s$ such that $f_p^\xi(x) = q_s$ for every $x\in \NN_s$).
 \end{itemize}
 We say that $q\leq p$ if and only if
 \begin{itemize}
 \item[(a)] $N_q\supset N_p$,
 \item[(b)] $F_q\supset F_p$, 
 \item[(c)] $f_q^\xi \supset f_p^\xi$ for every $\xi\in F_p$.
 \end{itemize}
 Similarly as in \cite[Lemma 4.3]{Koszmider} we prove that $\mathbb P$ is ccc; hence, it preserves cofinalities and cardinals \cite[Theorem IV.7.9]{kunen}. Finally, similarly as in \cite[Proposition 4.4]{Koszmider}, we prove that $\mathbb P$ forces that there are functions $(f_\xi\setsep \xi<\omega_1)$ with the properties indicated above.
 
As we have mentioned above, the modification is in replacing the interval $[0,1]$ by the compact space $2^\omega$ in the construction from \cite{Koszmider}; more precisely, it is a resolution given by the functions $(f_\xi)$ in the sense of \cite{wat}. In order to describe some more details, we recall the concept of a resolution below (we use the concept of a resolution from \cite{wat}; constructed compact spaces are easily seen to be homeomorphic to the ones considered in \cite{Koszmider}). Using verbatim the same arguments as in the original construction from \cite{Koszmider}, we may easily see that the constructed compact space $K$ has weight $\omega_1$, is first countable, hereditary separable, hereditary Lindel\"of, and that $\C(K)$ does not contain an uncountable equilateral set. It remains to show that our modification of Koszmider's example is $1$-dimensional.

 Let $L$ be a compact space, $B\subset L$ and for every $b\in B$ let us have a continuous function $f_b:L\setminus\{b\}\to [-1,1]$. By a \emph{resolution given by functions $(f_b)_{b\in B}$} we understand the space $K = R(L, (f_b)_{b\in B}) = (B\times [-1,1])\cup (L\setminus B)$ with the topology given by the following neighborhood basis. If $x\in L\setminus B$, then its neighborhood basis is the collection of all sets
 \[
 	\U(x,U):=\big((U\cap B)\times [-1,1]\big)\;\cup\; (U\setminus B),
 \]
 where $U$ is an open neighborhood of $x$ in the space $L$. If $x\in B$ and $y\in [-1,1]$, then the neighborhood basis at $(x,y)$ is the collection of all sets
 \[
 	\U(x, U, V):=(\{x\}\times V)\cup \big((U\cap f_x^{-1}(V)\cap B)\times [-1,1]\big)\cup (U\cap f_x^{-1}(V)\setminus B),
 \]
 where $U$ is an open neighborhood of $x$ in the space $L$ and $V$ is an open neighborhood of $y$ in the space $[-1,1]$.
 
 Finally, we prove the following proposition which yields that the above described modification gives a $1$-dimensional compact space.
 
 \begin{proposition}Let $L$ be a zero-dimensional compact space with countable character, $B\subset L$ and let $f_b:L\setminus\{b\}\to [-1,1]$ be a continuous function for every $b\in B$. Then the resolution $K=R(L, (f_b)_{b\in B})$ is a compact space with $\ind(K)\leq 1$. 
 \end{proposition}
 \begin{proof}It is well-known that $K$ is a compact space \cite[Theorem 3.1.33]{wat}. We will find a neighborhood basis at every point in $K$ such that the boundary 
 of each of its members is finite (in particular zero-dimensional). Let $x\in L\setminus B$ first. Then the set $\U(x, U)$ is a clopen neighborhood of $x$ in $K$ for a clopen neighborhood $U$ of $x$ in $L$. Moreover, sets of this type form a local neighborhood basis at $x$ in $K$.
 
 On the other hand suppose that $(x,y)\in B\times [-1,1]$ and let $\U(x, U, V)$ be a given neighborhood of $(x,y)$ in $K$. We want to find a smaller neighborhood of $(x,y)$ in $K$ whose boundary is finite. We may assume that $U$ is a clopen set.
 Let $W=[a,b]$ be a neighborhood of $y$ in $[-1,1]$ such that $W\subset V$ and let $W_n$ be open subsets of $[-1,1]$ such that $W_{n+1}\subset W_n$ for every $n\in\omega$, $W=\bigcap_{n\in\omega} \overline{W_n}$ and $W_0=V$. We claim that there exists a clopen set $C$ in $L\setminus \{x\}$ such that $f_x^{-1}(W)\subset C\subset f_x^{-1}(V)$ and $x\notin \overline{f_x^{-1}([-1,1]\setminus W_n)\cap C}$ for every $n\in\omega$. Indeed, let $\{B_n\colon n\in\omega\}$ be a local neighborhood basis at $x$ formed by clopen sets in $L$ with $B_{n+1}\subset B_n$ for each $n\in\omega$ and $B_0 = L$. One can easily find a clopen set $C_n\subset L\setminus \{x\}$ such that $f_x^{-1}(W)\subset C_n\subset f_x^{-1}(W_n)$. Without loss of generality we may suppose that $C_0\supset C_1\supset \dots$.
 Set $C=\bigcup_{n\in\omega}(C_n\cap B_n\setminus B_{n+1})$.

 It follows that the set $M=((C\cap U\cap B)\times [-1,1])\cup(C\cap U\setminus B)\cup (\{x\}\times W)$ is a neighborhood of $(x,y)$ in $K$. We shall prove that its boundary is a subset of $\{(x,a), (x,b)\}$, hence it is finite.
 
 First, every point of $L\setminus B$ as well as every point $(x^\prime, y^\prime)\in (B\setminus \{x\}\times [-1,1])$ is either an interior point of $M$ or a point outside of the closure of $M$. Moreover, for $z\in W\setminus\{a,b\}$ we have that $(x,z)$ is in the interior of $M$ because $(x,z) \in \U(x,U,(a,b)) \subset M$.
 
Finally, for $z\in [-1,1]\setminus W$ the point $(x,z)$ is outside of the closure of $M$. Indeed, there is $n\in\omega$ such that $z\notin \overline{W_n}$.
Let $U^\prime$ be an open neighborhood of $x$ in $L$ such that $U^\prime\cap \overline{f_x^{-1}([-1,1]\setminus W_n)\cap C}=\emptyset$. Let $V^\prime$ be an open neighborhood of $z$ disjoint from $W_n$. Then $\U(x, U^\prime, V^\prime)$ is an open neighborhood of $(x,z)$ disjoint from $M$.
\end{proof}

\subsection*{Acknowledgements}
We would like to thank an anonymous referee for careful reading of the paper and numerous valuable remarks. We would like also to thank O.~Kalenda for suggesting Lemma~\ref{l:contImageOfValdivia}.

\end{document}